\documentclass[11pt]{amsart}
\usepackage{amsfonts,amsmath,amssymb,amsthm}
\usepackage[latin1]{inputenc}
\theoremstyle{plain} 
\newtheorem{thm}{Theorem}

\newtheorem{lem}[thm]{Lemma}
\newtheorem{prop}[thm]{Proposition}

\theoremstyle{definition}

\theoremstyle{remark}

\newcommand{\thmref}[1]{Theorem~\ref{#1}}

\newcommand{\lemref}[1]{Lemma~\ref{#1}}

\newcommand{\x}{\mathbf{x}}

\ifx\undefined\xspace \usepackage{xspace} \fi
\newcommand{\qr}[1]{\eqref{#1}}

\newcommand{\ZZ}{{\mathbb Z}}

\newcommand\PP{{\mathbb P}} 

\newcommand{\ifemptythenelse}[3]{%
\def\epty{} %
\def\atempempty{#1} %
\ifx\atempempty\epty %
#2 %
\else %
#3 %
\fi
}
\newcommand{\ifnotempty}[2]{\ifemptythenelse{#1}{}{#2}}
\newcommand{\eptypar}[1]{\ifnotempty{#1}{\parens{#1}}}
\newcommand{\gparens}[3]{{\left#1 #2 \right#3}}
\newcommand{\parens}[1]{\gparens({#1})}

\newcommand{\hakparens}[1]{\gparens[{#1}]}

\newcommand{\probopi}[3]{{#1#2}{\hakparens{\,#3\,}}}

\renewcommand{\d}{\,{d}}

\newcommand{\Ex}[2][]{\probopi{\mathbb{E}}{#1}{#2}}

\newcommand{\Ordo}[1]{{O\eptypar{#1}}}
\newcommand{\ordo}[1]{{o\eptypar{#1}}}

\numberwithin{equation}{section}
\numberwithin{thm}{section}

\newsavebox{\Chai}
\sbox{\Chai}{\raisebox{0.3ex}[1.7ex][0.75ex]{{$\chi$}}}

\newcommand{\B}{B}
\newcommand{\CB}{\mathcal{B}}
\newcommand{\CM}{\mathcal{M}}

 \newcommand{\E}{E}

\renewcommand{\x}{x}
\newcommand{\CF}{\mathcal{F}}

\newcommand{\var}{\operatorname{var}}

\let\tl=\tilde
\let\mc=\mathcal
\let\X=X
\let\x=x
\let\T=T
\renewcommand{\L}{\mathcal{L}}

\newlength{\ppph}
\newcommand{\ii}[1]{^{(#1)}}

\renewcommand{\l}{l}

\def\ppp#1\par{\par\paragraph{#1}}
\begin{document}
\title{Unique Bernoulli $g$-measures} \author{Anders Johansson, Anders
  \"Oberg and Mark Pollicott}
\address{Anders Johansson\\ Division of Mathematics and Statistics\\
  University of G\"avle\\ SE-801 76 G\"avle\\ Sweden\\ and \newline
  Department of Mathematics\\ Uppsala University\\ P.O. Box 480 \ \
  SE-751 06 Uppsala \\ Sweden} \email{ajj@hig.se} \address{Anders
  \"Oberg\\ Department of Mathematics\\ Uppsala University\\ P.O. Box
  480 \ \ SE-751 06 Uppsala \\ Sweden} \email{anders@math.uu.se}
\address{Mark Pollicott\\ Mathematical Institute\\ University of Warwick\\
  Coventry CV4 7AL\\ UK} \email{mpollic@maths.warwick.ac.uk}
\date{\today} \keywords{Bernoulli measure, $g$-measure, chains with
  complete connections} \subjclass[2000]{Primary 37A05, 37A35, 60G10}

\maketitle
\begin{abstract}
  We improve and subsume the conditions of Johansson and \"Oberg
  \cite{johob} and Berbee \cite{berbee} for uniqueness of a
  $g$-measure, i.e., a stationary distribution for chains with
  complete connections.  In addition, we prove that these unique
  $g$-measures have Bernoulli natural extensions. In particular, we
  obtain a unique $g$-measure that has the Bernoulli property for the
  full shift on finitely many states under any one of the following
  additional assumptions.
  \begin{enumerate}
  \item\label{first}
$$\sum_{n=1}^\infty \left(\var_n \log g \right)^2<\infty,$$
\item\label{second} For any fixed $\epsilon>0$,
$$\sum_{n=1}^\infty e^{-\left(\frac{1}{2}+\epsilon \right)
  \left(\var_1 \log g+\cdots+\var_n \log g \right)}=\infty,$$
\item\label{third}
$$\var_n \log g=\ordo{\frac{1}{\sqrt{n}}}, \quad n\to \infty.$$
\end{enumerate}
That the measure is Bernoulli in the case of \eqref{first} is new.  In
\eqref{second} we have an improved version of Berbee's condition
(concerning uniqueness and Bernoullicity) \cite{berbee}, allowing the
variations of $\log g$ to be essentially twice as large. Finally,
\eqref{third} is an example that our main result is new both for
uniqueness and for the Bernoulli property.

We also conclude that we have convergence in the Wasserstein metric of
the iterates of the adjoint transfer operator to the $g$-measure.
\end{abstract}

\section{Introduction}\label{s.intro}\noindent
Let $S$ be a countable set. Let $\ZZ_+=\{0,1,2,\ldots\}$,
$\ZZ=\{\ldots,-1,0,1,2\ldots\}$, $X=S^\ZZ$, $X_+=S^{\ZZ_+}$ and
$X_-=S^{\ZZ\setminus\ZZ_+}$.  Any bi-infinite sequence $x\in \X$ and
$n\in \ZZ$, gives a one-sided infinite sequence
$x^{(n)}=(x_{-n},x_{-n+1},\ldots)$ in $\X_+$. Moreover, the stochastic
process $\{x^{(n)}\}_{n\in\ZZ}$ has the Markov property for any
distribution of $x$ in $\CM(X)$, where $\CM(X)$ denotes the Borel
probability measures on $X$, with respect to the product topology on
$X$.

Let $g \geq 0$ be a continuous function on $X_+$ such that
\begin{equation}\label{prob}
  \sum_{x_0\in S} g(x_0 x)=1, \, x\in X_+.
\end{equation}
A distribution $\mu\in\CM(\X)$ of $x\in\X$ is a \emph{$g$-chain} if
\begin{equation}
  \mu \left(x^{(n)}| x^{(n-1)}\right)=g\left(x^{(n)} \right)
\end{equation}
for all $n\geq 0$.  Thus, the process depends on the past according to
the $g$-function. Note that the distribution of a $g$-chain is
uniquely determined by the distribution $\mu\circ (x\ii0)^{-1}
\in\CM(\X_+)$ of its ``initial'' value $x\ii 0$.

If $g$ depends only on the choice of the new state then we have an
{\em i.i.d.}\ process, and if $g$ depends on the new state and the
previous one, then we have a Markov chain on the countable set $S$. If
we have dependence on the $k$ previous states, before moving to the
new state, we have a $k$-chain, and if there is no such restriction on
the dependence, we have a chain with complete, or infinite,
connections.

In this paper, we will restrict our attention to the case when $S$ is
a finite set and $g>0$. A stationary measure for our process is
sometimes called a $g$-measure, see Keane \cite{keane}, who introduced
this notion in ergodic theory. Important contributions were also
provided by Ledrappier \cite{ledrappier}, where in particular it was
shown that $g$-measures are equilibrium states, and Walters
\cite{walters1}, where the theory of $g$-measures was connected with
the transfer operator theory for general potentials. The theory has
also had a long, but slightly different appearance in the probability
theory of chains with complete connections, see e.g.\ Doeblin and
Fortet (1937) \cite{doeblin}, where it was proved that uniqueness of
$g$-measures follows from summable variations, and the works by
Iosifescu and co-authors, for instance that with Theoderescu in
\cite{iosifescu3} and with Grigorescu \cite{iosifescu2}. The theory is
also connected to that of iterated function systems, or iterated
random functions; see Diaconis and Freedman \cite{diaconis} and the
references therein. A recent contribution by Iosifescu is
\cite{iosifescu1}.  We have not attempted to give a complete survey of
the literature, but rather to point the reader in some important main
directions of the different appearances of the problems we are
considering here.

If $T$ is the left shift map on $\X_+$, then a $g$-measure can
alternatively be viewed as $\T$-invariant probability measure
$\mu\in\CM(\X_+)$, with the property that $g = {d\mu}/{d(\mu\circ
  T)}$. Since $X_+$ is compact due to the finiteness of $S$, it
follows that there always exists a $g$-measure.  Uniqueness is however
not automatic, as was clarified by Bramson and Kalikow in
\cite{bramson}. Examples of non-uniqueness have since then been
provided in, e.g., \cite{berger} and \cite{hulse}.

A useful way of viewing a $g$-measure is as a fixed point of the dual
${\mathcal \L}^*$ of the transfer operator ${\mathcal \L}$, defined
pointwise by
$${\mathcal \L}f(x)=\sum_{Ty=x} g(y)f(y),$$
where ${\mathcal \L}: C(X_+)\to C(X_+)$. Hence, a $g$-measure can be
viewed as a probability measure satisfying ${\mathcal \L}^* \mu=\mu$.
 
If we do not impose the probability assumption \eqref{prob}, the
eigen-measure of the dual of the transfer operator is not invariant in
general, but may instead look for eigen-measure solutions $\nu$ of
${\mathcal \L}^* \nu=\lambda \nu$, where $\lambda>0$ is the greatest
eigenvalue of the unrestricted transfer operator ${\mathcal L}$,
$${\mathcal \L}f(x)=\sum_{Ty=x}e^{\phi}(y)f(y),$$
where $\phi$ is the potential function, usually belonging to a
function space with the same regularity conditions as the test
functions $f$.

In this paper our results only concern the case of probabilistic
weight functions, that is $\phi=\log g$, where $g$ satisfies
\eqref{prob}.  In \cite{johob}, it was proved that there exists a
unique $g$-measure if $g>0$ and
\begin{equation}\label{square}
  \sum_{n=1}^{\infty} \left(\var_{n} \log g\right)^2<\infty,
\end{equation}
where the $n$th variation of a function $f$ is defined as
$$
\var_{n} f = \sup_{x\sim_n y} |f(x) - f(y)|,
$$ 
where $x\sim_n y$ means that $x$ and $y$ coincide in the first $n$
coordinates.

This condition of {\em square summability of variations} of the
$g$-function for the $g$-chain is proven \cite{berger} to be sharp, in
the sense that for all $\epsilon>0$ there exists a $g$-function such
that
$$\sum_{n=1}^{\infty} \left(\var_{n} \log g\right)^{2+\epsilon}<\infty,$$
with more than one $g$-measure. This should be compared to an older
result of Dyson \cite{dyson} for general potentials $\phi$,
identifying {\em summability of variations} as sharp, in the sense
that we may have multiple eigen-measure solutions of ${\mathcal
  \L}\nu=\lambda \nu$, when
$$\sum_{n=1}^\infty \left(\var_n \phi \right)^{1+\epsilon}<\infty.$$ 
In view of this dichotomy in terms of summability of powers of
variations, Berbee's two results from the late 1980s are intriguing.
He proves uniqueness of a $g$-measure and of an eigen-measure in the
general case, when
\begin{equation}\label{berbee}
  \sum_{n=1}^\infty e^{-r_1-\cdots-r_n}=\infty,
\end{equation}
where $r_n=\var_n \log g$ or $r_n=\var_n \phi$, respectively.  This
allows for the non-summable sequence $r_n=\frac{1}{n}$. In the case of
general potentials this is sharp, modulo a constant factor, see
\cite{aizenman}, but obviously not for $g$-measures, since square
summability of variations cover sequences
$r_n=\frac{1}{n^{1/2+\epsilon}}$, $\epsilon>0$.

Since it was shown in \cite{johob} that there are sequences that
satisfy Berbee's condition but not square summability, it becomes
interesting in the case of proving uniqueness of a $g$-measure to ask
if there is a condition that subsumes in a natural way these two
uniqueness conditions.  We provide conditions for uniqueness that
contains both square summability of variations and Berbee's condition
for a unique $g$-measure.

Our method of proof also allows us to conclude that the unique
$g$-measure is {\em Bernoulli}, meaning that if we look at the {\em
  natural extension} of the dynamical system, i.e.,
$$x^{(n)}=(x_{-n},x_{-n+1},\ldots),$$
$n\geq 0$, with the $g$-measure $\mu$ as initial distribution for
$x^{(0)}$, then this stochastic process is isomorphic to an {\em
  i.i.d.}\ process.

The Bernoulli property was also proved by Berbee, but is new for
square summability of variations (convergence for the iterates of the
transfer operator is known from \cite{johob1}). For instance we prove
that we have a {\em unique $g$-measure} that is furthermore {\em
  Bernoulli} under the following three special conditions:
\begin{enumerate}
\item\label{first}
$$\sum_{n=1}^\infty \left(\var_n \log g \right)^2<\infty;$$
\item\label{second} For any fixed $\epsilon>0$,
$$\sum_{n=1}^\infty e^{-\left(\frac{1}{2}+\epsilon \right)
  \left(r_1+\cdots+r_n\right)}=\infty;$$
\item\label{third}
$$\var_n \log g=\ordo{\frac{1}{\sqrt{n}}}, \quad n\to \infty.$$
\end{enumerate}
The last example is in a sense the weakest condition we have for a
unique Bernoulli $g$-measure. The second is an improvement of Berbee's
condition with a constant, owing to our method. For other results
concerning the Bernoulli property for $g$-measures and equilibrium
states for general potentials, see \cite{walters3}.

It would be interesting to investigate whether there is a sharp
constant so that we have uniquness and perhaps the Bernoulli property
for $\var_n \log g\leq \frac{c}{\sqrt{n}}$. Perhaps the $\leq$ should
be replaced by a $<$ and perhaps the constants are different for
uniqueness and for the Bernoulli property.

Our method of proof relies on two main ideas.

Firstly, we use a forward block coupling, including solving the
renewal equation to obtain an estimate of the probability of having
conflicts between two extensions of a $g$-chain, starting from two
different distributions. This argument is then applied to a
perturbation of one of the extensions to a sequence of $g$-functions
corresponding to a sequence of Bernoulli measures that converges in
the $\bar{d}$-metric to the unique $g$-measure under investigation.

Secondly, we use Hellinger integral estimates from \cite{jacod} to
calculate the probability of not having a conflict (that is, different
entries in a corresponding coordinate) in the extensions of two
initial distributions when we add a new block of positive integer
length $b_\l$ (at a certain height $\l \geq 1$ in the extension). We
show that if these probabilities are $e^{-\rho_\l}$, the maximal
probability of not having a conflict, as defined through the total
variations distance, then we can approximate $\rho_\l$ in such a way
that it asymptotically includes a square sum of the variations, where
the sums are taken over the increasing blocks. More precisely, if we
define recursively an increasing sequence of natural numbers
$B_\l=B_{\l -1}+b_\l$, $\l \geq 1$, $B_0=0$, we get the estimate
$$\rho_\l \leq (1+\ordo1) s_\l, $$ 
where
$$
s_\l := \sum_{k=B_{\l-1}}^{B_\l -1} \frac{1}{8} (\var_k \log g)^2.
$$
Finally, we define
$$r_\l = \sqrt{2s_\l} + 2 s_\l. $$
In the special cases (1) and (3) above, we have found examples of
exponential increase of $b_\l$ in $\l$. If $b_\l=1$ for all $\l \geq
1$, we obtain Berbee's situation, in which case $\rho_\l\leq
r_l=\var_\l \log g$. However our estimates show that although this is
of the right order, our method allows one to improve Berbee's result
by a constant; essentially, the variations are allowed to be twice as
big.

We can now state one version of our main result.
\begin{thm}
  We obtain a unique $g$-measure which is Bernoulli, if there is a
  sequence of positive integers $\{ b_\l \}_{l=1}^\infty$ such that,
  with $\{r_\l\}$ defined from $\{b_\l\}$ as above, $\limsup r_\l =0$
  and
  $$ \sum_{\l=1}^\infty b_\l \, e^{-r_1-\dots-r_{\l}} = \infty.$$ 
\end{thm}

{\bf Acknowledgement.} We would like to thank Jean-Pierre Conze for
valuable discussions. In addition, we would like to acknowledge the
hospitality of the Institut Mittag-Leffler, where this paper was
completed.

\section{Preliminaries}

\subsection{The Bernoulli property and the $\bar d$-metric}

\let\tp=\hat Let $\CM^g(\X) \subset \CM(\X)$ denote the set of
$g$-chains corresponding to the $g$-function $g$, i.e.\ the set of
$\mu$ such that
$$ \mu \circ (x\ii n)^{-1} = \L^{*n}[\mu\circ (x\ii 0)^{-1}]. $$ 
Let $\CM_T^g(X)$ denote the set of $g$-measures.

On $\CM(X_+)$ we have the natural filtration $\{\CF_n\}$ of the Borel
$\sigma$-algebra, where $\CF_n=\sigma(x_0,\dots,x_{n-1})$. For a
measure $\nu\in \CM(X_+)$ and a sub $\sigma$-algebra $\CB\subset \CF$,
we let $\nu\vert_{\CB}$ denote the restriction to $\CB$.

Recall that \emph{coupling} (or \emph{joining}) between two
probability distributions $\mu\in\CM(X,\CF)$ and $\tp\mu\in\CM(\tp
X,\tp\CF)$ is a probability distribution $\nu\in\CM(X\times
Y,\CF\otimes \tp\CF)$ of a pair $(x,\tp x)\sim X\times \tp X$ such
that the marginals are are given by $x\sim \mu$ and $\tp x\sim
\tp\mu$.  For a pair of probability measures $(\mu,\tp\mu)$ on the
measure space $\CM(X,\CF)$, where $X=S^\ZZ$ and $\CF$ denotes the
corresponding product $\sigma$-algebra, let
\begin{displaymath}
  \bar d(\mu,\tp\mu) := 
  \inf_\nu \limsup_{n\to\infty } \nu\{x_{-n}\not= \tp x_{-n} \},
\end{displaymath}
where the infimum is taken over all couplings $\nu$ between $\mu$ and
$\tp\mu$.  This corrsponds to the $\bar d$-metric introduced by
Ornstein (for a reference, see e.g., \cite{quasc} or
\cite{ornsteinb}), if we take the restriction to the space $\CM_T(X)$
of shift invariant measures; on $\CM(X)$ it is a pseudo-metric. Notice
that in our case, the definition of $\bar d$ uses couplings that are
not necessarily translation invariant even if the marginals are. In
\cite{quasc}, the authors define $\bar d$ on $\CM_T(X)$ by taking the
infimum over couplings that are invariant under the transformation
$T\times T$ on $X\times X$. However, the original definition by
Ornstein does not presuppose translation invariant couplings.

An invariant measure $\mu\in\CM_T(X)$ is \emph{Bernoulli} if it can be
realised by an isomorphism with a Bernoulli shift. In other words,
there is a bijectively measurable mapping $\phi:A^{\ZZ}\to X$ such
that $\phi\circ T'=T\circ \phi$, where $T'$ denote the shift on
$A^\ZZ$ and such that $\mu=\mu'\circ \phi^{-1}$ where $\mu'$ is a
Bernoulli shift, which means that, under $\mu'$, each symbol is chosen
independently according to some fixed discrete probability on the
finite set $A$. Ornstein proves in \cite{ornsteinb} that the set $\CB$
of measures in $\CM_T(X)$ having the Bernoulli property is closed in
the topology induced by the $\bar d$-metric. Many classes of
$g$-functions are well-known to give rise to unique $g$-measures with
the Bernoulli property. In particular, if the $g$-function is
determined by a finite number of coordinates, i.e., it is the
transition probabilities for $N$-chains, for some finite $N$; see
e.g.\ \cite{ornsteinb} or \cite{quasc}. We also remind the reader of
the results of Walters, see \cite{walters3}.

It easy to see that any given $g$-function $g$ with $\var_N\log g\to
0$ as $N\to\infty$ can be arbitrarily well approximated by finitely
determined $g$-functions, e.g.\ let $\tp g_N(x) = g(x_0, x_1, \dots,
x_N z)$, for a fixed $z\in\X_+$, whence
$$
\|\log \tp g_N - \log g\|_\infty \leq \var_N \log g.
$$
 
Let $\mu$ and $\tp\mu$ denote $g$-chains corresponding to the
$g$-functions $g$ and $\tp g$, respectively.  Our strategy --- which
is similar to that in used in \cite{quasc} --- for proving that the
$g$-measure $\mu$ is Bernoulli, is first to show that, the $\bar
d$-distance between $\mu$ and $\tp\mu$ can be bounded by a function
which is continuous in $s=\|\log g - \log\tp g\|_\infty$ and that
fixes zero.

A finite \emph{block-structure} is a sequence $\{b_\l\}_{\l=1}^M$ of
positive integers $b_\l\geq 0$. We refer to the index $\l$ as
\emph{levels}.  By a \emph{block-variation pair}, we mean a
block-structure $\{b_\l\}$ in conjunction with a sequence $\{r_\l\}$
of positive real numbers.  For a block-variation pair
$(\{r_\l\},\{b_\l\})=(\{r_\l\}_{\l=1}^M,\{b_\l\}_{\l=1}^M)$ we define
a real number
\begin{equation}
  \label{e.deltadef}
  \bar\delta(\{r_\l\},\{b_\l\}) := 
  \frac
  {1+\sum_{\l=1}^M b_\l e^{-r_1-\dots-r_{\l-1}}\,(1-e^{-r_\l})}
  {\sum_{\l=1}^M b_\l e^{-r_1-\dots-r_{\l-1}}}, 
\end{equation}
where for simplicity we have adopted the convention that
$e^{-r_1-\dots-r_{\l-1}}=1$ for $\l=1$.  A block-variation function
$r$ associates a positive real number $r(B,b)$ to integers $B\geq 0$
and $b>0$. Given a block-structure $\{b_\l\}$ and a block-variation
function $r$, we define the corresponding sequence $\{r_\l\}$ by
setting
\begin{equation}
  \label{e.rldef}
  r_\l := r(b_1+b_2+\dots+b_{l-1},b_\l). 
\end{equation}
In this context, we will denote the pair $(\{r_\l\},\{b_\l\})$ by
$(r,\{b_\l\})$.

Our first lemma establishes a bound on the $\bar d$-metric between
$g$-chains which is continuous in the supremum norm.
\begin{lem}\label{l.cont}
  Let $g$ and $\mu$ be as above.  There is a block-variation function
  $\rho^g(B,b)$, such that for any block-variation pair
  ($\{r_\l\},\{b_\l\})$ satisfying
  \begin{equation}
    \rho^g_\l\leq r_\l\label{e.valid}
  \end{equation}
  we have
  \begin{equation}\label{e.ineq}
    \bar d(\mu,\tp\mu) \leq \bar\delta\left(\{r_\l+s\cdot b_\l\},\{b_l\}\right),
  \end{equation}
  for all $g$-chains $\tp\mu$ corresponding to a $g$-functions $\tp g$
  with
$$ \| \log g - \log\tp g\|_\infty = s. $$
\end{lem}
We say that pairs $(\{r_\l\},\{b_\l\})$ satisfying \qr{e.valid} are
\emph{valid} for $g$.  We prove this lemma in the next subsection.
Note that, for a fixed finite pair
$(\{r_\l\}_{\l=1}^M,\{b_\l\}_{\l=1}^M)$, the quantity
$\bar\delta(\{r_\l\},\{b_\l\})$ is clearly continuous in $\{r_\l\}$ so
that in particular
$$ 
\lim_{s\to 0+} \bar\delta(\{r_\l+sb_\l\},\{b_\l\}) =
\bar\delta(\{r_\l\},\{b_\l\}).
$$

To see how we can deduce the the Bernoulli property, notice that if
\begin{equation}
  \inf_{\{r_\l\}, \{b_l\}} \bar\delta(\{r_\l\},\{b_l\}) = 0,\label{e.znull}
\end{equation}
where the infimum is taken over all pairs $(\{r_\l\},\{b_l\})$ that
are valid for $g$. Then, for every $\epsilon>0$, we can find a
block-structure $\{b_l^\epsilon\}_{l=1}^M$ with
$\bar\delta(g,\{b_l^\epsilon\})<\epsilon$. By the continuity of
$\bar\delta(\cdot,\{b_\l^\epsilon\})$ we can take a finitely
determined (locally constant) $g$-function $\tp g$ with $g$-measure
$\tp\mu$ such that
$$   \bar d(\mu,\tp\mu) \leq \bar\delta\left(r+\| \log g -
  \log\tp g \|_\infty,\{b_l^\epsilon \}\right) < 2\epsilon,
$$ 
say.  It follows that the $\bar d$-distance between the $g$-measure
$\mu$ of $g$ and the set $\CB$ of Bernoulli measures is zero and since
$\CB$ is closed with respect to the $\bar d$-distance
\cite{ornsteinb}, we conclude that $\mu\in\CB$. Moreover, it is
well-known and easy to see that this $g$-measure corresponding to $g$
must be unique. We collect the conclusions in the following Theorem.
\begin{thm}\label{l.bernoulli}
  If \qr{e.znull} holds then we have a unique Bernoulli $g$-measure
  $\mu$ corresponding to $g$. Moreover, $\mu$ is attractive in the
  sense that $\L^{*n}\nu$ converges weakly to $\mu$ for any initial
  distribution $\nu\in\CM(\X_+)$.
\end{thm}
We prove the last statement in Section \ref{s.xxx}.

\subsection{The coupling argument and the proof of \lemref{l.cont}}

In order to obtain the bound in \qr{e.ineq}, we will need to construct
a coupling between a $g$-chain $\mu$ and a $\tp g$-chain $\tp\mu$, by
defining the two chains $x\sim\mu$ and $\tp x\sim \tp\mu$ on the same
probability space $(\Omega,\mc F, \PP)$. Assume that $s=\|\log g -
\log \tp g\|_\infty$.  The distributions of $x\ii 0$ and $\tp x\ii0$
are arbitrary.

The coupling we construct uses a block-structure $\{b_\l\}$, where we,
at certain times $n$, extend the two $g$-chains with block of symbols
of length $b_\l$ until we reach a conflict --- i.e.\ a coordinate with
different symbols --- in the extension.  Extending the two chains
$x\ii n$ and $\tp x\ii n$ with a block of length $b_\l$, means
specifying a distribution of the pair $(x\ii {n+b_\l}, \tp x\ii
{n+b_\l})$ such that $x\ii{n+b_\l}$ has distribution
$\L^{*b_\l}_g\delta_{x\ii n}$ and $\tp x\ii{n+b_\l}$ has distribution
$\L^{*b_\l}_{\tp g}\delta_{\tp x\ii n}$.  We are at \emph{level} $\l$
when we extend with a $b_\l$-block and this presupposes, that
previously, without conflict, we have extended with blocks at levels
$0,1,\dots,\l-1$ of a total length
$$B_{\l-1}=b_1+b_2+\dots+b_{\l-1}. $$ 

For $(y,\tp y) \in \X_+\times\X_+$, define the {\em concordance time}
as the non-negative integer
$$
\kappa(y,\tp y) = \sup\{ k\geq 0: y\sim_k \tp y\}.
$$
The event of success (or ``no conflict'') means that that
$$ \kappa(x\ii {n+b_\l},\tp x\ii{n+b_\l}) = \kappa(x\ii n,\tp x\ii n)+b_\l. $$
We always use a \emph{maximal coupling} between the chains, i.e., a
coupling that makes the probability of success maximal.

We show \qr{e.ineq} in \lemref{l.cont}, by defining on the same
probability space $(\Omega,\mc F, \PP)$ a Markov chain $Y_n$ taking
values in $\ZZ$.  Given a block-variation pair $(\{r_\l\},\{b_\l\})$,
we define an associated Markov chain $Y_n = Y_n^{\{r_\l\},\{b_\l\}}$,
$n\geq 0$, as follows: Let $Y_0=0$. If $Y_n\not=B_\l$ for some $\l$,
simply let $Y_{n+1}=Y_n+1$, but, if $Y_n=B_{\l-1}$ for some
$\l=1,\dots,M$ then
\begin{equation}
  \label{e.Y}
  Y_{n+1} = 
  \begin{cases}
    B_{\l -1} + 1 & \text{ with probability $e^{-r_{\l}}$ }\\
    -b_{\l} & \text{ with probability $1-e^{-r_{\l}}$ }.
  \end{cases}
\end{equation}
If $Y_n=B_M$ we set $Y_{n+1}=0$, because we want to avoid to have
infinite waiting time in mean when we later solve the renewal
equation.

By using the Renewal Theorem, we show in section \ref{s.xxx} the
following.
\begin{lem}\label{l.renewal}
  Assume that the Markov chain $Y_n$ is defined from parameters $r$
  and $\{ b_\l \}$ as in \qr{e.Y}. Then
$$ 
\limsup_{n\to\infty} \PP\{Y_n\leq 0\} \leq \bar\delta(r,\{b_\l\})
$$
where $\bar\delta$ is defined in \qr{e.deltadef}.
\end{lem}

We couple the Markov chain $Y_n=Y_n^{\{r_\l+sb_\l\},\{b_\l\}}$ with
the block-extensions such that, for all $n$,
\begin{equation}
  \label{e.domina}
  \kappa(x\ii n,\tp x\ii n) \geq Y_n. 
\end{equation}
Since $x_{-n}\not=\tp x_{-n}$ precisely when $\kappa(x\ii n,\tp x\ii
n)=0$ it then becomes clear from \lemref{l.renewal} that
\begin{equation}
  \bar d(\mu,\tp\mu) \leq \limsup \PP\{ Y_n \leq 0 \} \leq
  \bar\delta(\{r_\l\},\{b_\l\}), 
\end{equation}
which is \qr{e.ineq} in \lemref{l.cont}.

We execute, at time $n$, a block-extension at level $\l$, precisely
when $Y_n = B_{\l-1}$.  In order to maintain \qr{e.domina}, we should
couple the transition of $Y_n$ so that $Y_n = -b_\l$ if the extension
is unsuccessful; then \qr{e.domina} holds up true to time $n+b_\l$
even if coordinates between $-n$ and $-n-b_\l$ should disagree. A
sufficient and necessary condition for the mechanism to work is
therefore that the probability that $Y_n$ of moves up one level, i.e.\
$e^{-r_\l}$, is less than the probability that the block-extension is
successful. We define $\rho^{g,\tp g}(B_{\l-l},b_\l)$ as the infimum,
over $(x\ii n,\tp x\ii n)$, of the probability of success, conditioned
on $(x\ii n,\tp x\ii n)$, under the restriction that $\kappa(x\ii
n,\tp x\ii n)\geq B_{\l-1}$.  More precisely, we need to show that,
the condition that $r_\l$ is valid implies that $r_\l+s\cdot b_\l$ is
less than $\rho^{g,\tp g}(B_{\l-1},b_\l)$. As before, we assume that a
maximal coupling is used. Notice that, if the extension is executed at
level $\l$, we have $\kappa(x\ii n,\tp x\ii n)\geq Y_n = B_{\l-1}$, by
\qr{e.domina}.

What remains to complete the proof of \lemref{l.cont} is to show that
$$ \rho^{g,\tp g}(B,b) \geq \rho^{g,g}(B,b) + s\cdot b, $$
and to give an explicit expression for $\rho^g := \rho^{g,g}$.

It is well-known that the probability for a successful extension in a
maximal coupling is given by the total variation metric between the
marginals of the extension, see e.g.\ \cite{lindvall}.  The success
probability is given by
\begin{equation*}
  \int \left(\frac{d\tp\eta}{\d\eta} \wedge 1\right) \d\eta
  = \left(1 - \frac 12 \cdot d_{TV}\left(\eta,\tp\eta\right) \right).  
\end{equation*}
In our situation we can identify the marginals $\eta$ and $\tp\eta$
with the distributions on $\CF_b$ given by
$$
\eta = \L_g^{*b}\delta_{x\ii n}\vert_{\CF_b}, \quad \tp\eta = \L_{\tp
  g}^{*b}\delta_{\tp x\ii n}\vert_{\CF_b}.
$$
for some $x\ii n$ and $\tp x\ii n$ that satisfy $\kappa(x\ii n,\tp
x\ii n) \geq B$. Let $\CM_{B,b}^{g,\tp g}$ denote the set of such
pairs $(\eta,\tp\eta)$.

We then define
\begin{equation}
  \rho^{g,\tp g}(B,b) := \sup\left\{ -\log 
    \int \left(\frac{d\tp\eta}{\d\eta} \wedge 1\right) \d\eta
    : (\eta,\tp\eta)\in\CM^{g,\tp g}_{B,b} \right\}.
  \label{e.rhodef}
\end{equation}
Notice that, since $\tp g/g \geq e^{-s}$, we have
\begin{equation}
  \label{e.rndef}
  \frac{d\tp\eta}{d\eta} = 
  \frac{\tp g(\tp x)\tp g(T\tp x)\cdots \tp g(T^{b -1}\tl x)}
  {g(x) g(Tx)\cdots g(T^{b -1}x)} 
  \geq 
  e^{-bs} \cdot 
  \frac{g(\tl x) g(T\tl x)\cdots  g(T^{b -1}\tl x)}
  {g(x) g(Tx)\cdots g(T^{b -1}x)} 
\end{equation}
and the right hand side equals $e^{-bs} \cdot {d\tl\eta}/{d\eta}$,
where
$$
\tl\eta := \L^{*b}_g\delta_{\tp x\ii n}.
$$
We then obtain from \qr{e.rhodef} that
\begin{equation}
  \label{e.untp}
  \rho^{g,\tp g}(B,b) \leq \rho^{g,g}(B,b) + s\cdot b, 
\end{equation}
where
\begin{equation}
  \rho^{g,g} = \sup\left\{ -\log 
    \int \left(\frac{d\tl\eta}{\d\eta} \wedge 1\right) \d\eta
    : (\eta,\tl\eta)\in\CM^{g,g}_{B,b} \right\}.\label{e.rho2}
\end{equation}
Since $\rho^g=\rho^{g,g}$, this concludes the proof of
\lemref{l.cont}.  \qed

\subsection{Estimates using Hellinger integrals}

In order to arrive at verifiable conditions that ensures that $\inf
\bar\delta(r,\{b_\l\})=0$, i.e.\ the assumption \qr{e.znull} in
\thmref{l.bernoulli}, we estimate the total variation metric using the
Hellinger integral. This was done in some special cases also in our
earlier paper \cite{johob2}.  Define the ``Hellinger block-variation''
$h(B,b) = h^g(B,b)$ by
\begin{equation}
  \label{e.hellvar}
  h^g(B,b) = \sup\left\{ 
    -\log H(\eta,\tl\eta):
    (\eta,\tl\eta)\in\CM^{g,g}_{B,b} 
  \right\} 
\end{equation}
where
\begin{equation*}
  \label{e.hellingerint}
  H(\eta,\tl\eta) = \int \left( \frac{d\tl\eta}{d\eta} \right)^{\frac 12} \!d\eta
\end{equation*}
is the Hellinger integral of $\eta$ and $\tl\eta$. We always have
$0\leq H \leq 1$.

The relevant estimates we will need are collected in the following
lemma.
\begin{lem}\label{l.bounds}
  We have the following relations between the block-variations defined
  above
  \begin{align}
    \label{e.rhoh1}
    \rho^g &\leq -\log\left( 1 - \sqrt{1- \exp(-2 h^g)}\right) \text{    and, in particular},\\
    \label{e.rhoh}
    \rho^g &\leq \sqrt{2 h^g} + 2 h^g, \\
    \label{e.hadd}
    h^g(B,b) &\leq \sum_{k=B}^{B+b-1} h^g(k,1), \\
    \intertext {As $k\to\infty$}
    \label{e.hvarlog}
    h^g(k,1) &= (1+\ordo 1) \frac 18 (\var_k \log g)^2, \\
    \intertext{and as $w\to 0$ }
    \label{e.rvarlog}
    \rho^g(B,b) &\leq \left(1+\Ordo {w}\right) \frac 12 \cdot w
  \end{align}
  where
 $$
 w=\sqrt{\sum_{k=B}^{B+b} (\var_k \log g)^2}.
 $$ 
\end{lem}

A condition ensuring that condition \qr{e.znull} is satisfied is given
in the following Theorem. We say that a block-variation
$(\{r_\l\},\{b_\l\})$ is \emph{eventually valid} if for some $l_0$, we
have $r_\l\geq \rho^g_\l$ for $\l\geq l_0$.
\begin{thm}\label{t.explicit}
  A sufficent condition for the conclusions of \thmref{l.bernoulli} to
  hold is that there is some infinite eventually valid block variation
  pair $\left(\{ r_\l\}_{l=1}^\infty,\{b_\l\}_{\l=1}^\infty\right)$
  such that $\limsup r_\l = 0$ and
  \begin{equation}
    \sum_{\l=1}^\infty e^{-r_1-\dots-r_{\l-1}}\, b_{\l} = \infty.
    \label{e.infsum}
  \end{equation}
\end{thm}
\begin{proof}
  We verify \eqref{e.znull}, that is, we show that
  \begin{equation}
    \inf_{\{r_\l\}_{\l=1}^M, \{b_l\}_{\l=1}^M} \frac
    {1+\sum_{\l=1}^M b_\l e^{-r_1-\dots-r_{\l-1}}\,(1-e^{-r_\l})}
    {\sum_{\l=1}^M b_\l e^{-r_1-\dots-r_{\l-1}}}=0.
  \end{equation}
  To see this, note that $(1-e^{-r_\l})\leq r_\l$. Hence, by the
  assumption \eqref{e.infsum} and since $r_\l\to 0$, as $\l \to
  \infty$, we have
$$\inf_{\{r_\l\}_{\l=1}^M, \{b_l\}_{\l=1}^M} \frac
{1+\sum_{\l=1}^M b_\l e^{-r_1-\dots-r_{\l-1}}\, r_\l} {\sum_{\l=1}^M
  b_\l e^{-r_1-\dots-r_{\l-1}}}=0,$$ and the conclusion follows.
\end{proof}

\subsection{Examples}

By setting $b_\l=1$ and noting that $r_\l = (1/2+\epsilon)\var_\l \log
g$ eventually dominates $\rho^g_\l$ by \qr{e.rvarlog}, we can deduce
the special case \eqref{second} in the Introduction. We now show the
results under the hypotheses in the the special cases \eqref{first}
and \eqref{third}, by verifying that the conditions in
\thmref{t.explicit} are satisfied.

Note that the following proposition gives a uniqueness result that is
not covered by earlier results, for instance in \cite{johob}.
\begin{prop}
  We have a unique $g$-measure with the Bernoulli property if
  $$ \var_n\log g = \ordo{\frac{1}{\sqrt{n}}}.$$ 
\end{prop}
\begin{proof}
  Take a real number $c>1$. Let $B_0=0$ and let
  $B_\l=\lceil{c^\l/(c-1)}\rceil$ for $\l\geq 1$, so that for $\l\geq
  2$ $b_l = B_l - B_{\l-1}$ satisfies
  $$
  b_\l\geq \left\lfloor{c^\l/(c-1)} - {c^{\l-1}/(c-1)} \right\rfloor =
  \lfloor c^\l \rfloor \geq 1.
  $$

  Define $r_\l$ by
  $$
  r^2_\l = \sum_{n=B_{\l-1}}^{B_{\l}-1} (\var_n\log g)^2.
  $$
  For $\l\geq 2$, we have by assumption that (as $\l\to\infty$)
  \begin{align*}
    r^2_\l &\leq \ordo 1\cdot \sum_{n=B_{\l-1}}^{B_{\l}-1} \frac1{\sqrt n}, \\
    &\leq \ordo 1 \cdot
    \int_{{c^{\l-1}/(c-1)}}^{{c^\l/(c-1)}} \frac{1}{x}\,dx \\
    &= \ordo{\log c} = \ordo{(\log c)^2}.
  \end{align*}
  The integral estimate of the partial sums of the harmonic series
  follows since $B_{\l-1} \geq {c^{\l-1}/(c-1)}$ and $B_\l-1 \leq
  {c^{\l}/(c-1)}$.

  Since, by \qr{e.rvarlog}, $\rho^g_\l \leq r_\l$ eventually, we can
  apply \thmref{t.explicit}.  We already know that $r_\l=\ordo{\log c}
  \to 0$ as $\l\to\infty$. Moreover, each term in the sum of
  \qr{e.infsum} can be estimated as
  \begin{equation}\label{e.term1}
    b_\l e^{-r_1-\cdots-r_\l} 
    \geq \exp\{ \l \log c - \l \cdot \ordo{\log c} \} \to \infty
  \end{equation}
  which verifies \qr{e.infsum}.
\end{proof}

We now show that the uniqueness condition of \cite{johob} also gives
the Bernoulli property.
\begin{prop}\label{finite}
  We have a unique $g$-measure with the Bernoulli property if
$$\sum_{n} (\var_n\log g)^2 < \infty.$$
\end{prop}
\begin{proof}
  First note that if $\{r_\l\}$ is a block-variation relative to
  blocks $\{b_\l\}$ such that
  $$ r_1+r_2+\dots < \infty, $$ 
  then it is clear that the conditions in \thmref{t.explicit} hold for
  $\{r_\l\}$ and $\{b_\l\}$.

  We define the blocks $B_\l$ such that $B_0=0$ and
  $$ 
  B_\l = \inf\left\{ B>B_{\l-1}: \sum_{n=B}^\infty (\var_n\log g)^2
    \leq L/2^\l\right\}
  $$
  where $L=\sum_{n=0}^\infty (\var_n\log g)^2$. Then with $r_\l$
  defined by
  $$r_\l^2 = \sum_{n=B_{\l-1}}^{B_\l-1} (\var_n\log g)^2,$$ 
  we have $r_{\l+1} \leq \Ordo{ \sqrt{L/2^\l} }$ and $\{r_\l\}$ is
  clearly a summable sequence since it decreases
  geometrically. Moreover, $\rho^g_\l \leq r_\l$ eventually by
  \qr{e.rvarlog}.
\end{proof}

\section{Remaining proofs}\label{s.xxx}

\subsection{Proof of \lemref{l.bounds}}

Note that \qr{e.rvarlog} is easily deduced from \qr{e.hvarlog} and
\qr{e.hadd}.

\begin{proof}[Proof of \qr{e.rhoh1} and \qr{e.rhoh}]
  In order to relate the two variation functions $\rho^g$ and $h^g$,
  we use the following bound (Proposition V.4.4 in \cite[p.\
  311]{jacod}) on the total variaton metric
  \begin{equation}
    \label{e.bounddtv}
    d_{TV}(\eta,\tl\eta) \leq 2 \sqrt{1-H(\eta,\tl\eta)^2}.
  \end{equation}
  This relation immediately gives \qr{e.rhoh1} by re-writing the
  relations in terms of $\rho^g$ and $h^g$. From this, we obtain
  \qr{e.rhoh} as a useful approximation by easy calculations. In the
  estimate \qr{e.rhoh}, the first term $\sqrt{2}\cdot \sqrt{h^g}$ is
  sharp ($\sqrt{2}$ is the sharp number), but the second, $2\cdot
  h^g$, is not. Slightly lower numbers than $2$ are possible.
\end{proof}

\begin{proof}[Proof of \qr{e.hadd}]
  Let $(\eta,\tl\eta) \in \CM^{g,g}_{B,b}$.  We can explicitly write
  \begin{equation}
    H(\eta,\tl\eta) = \int 
    \left(
      \frac
      {g(\tl x) g(T \tl x) \cdots g(T^{K-1} \tl x)}
      {g(x) g(T x) \cdots  g(T^{K-1} x)} \right)^{1/2} 
    \, d \eta(x), 
    \label{gprod}
  \end{equation}
  where $(x,\tl x)\in (X_+,X_+)$ satisfies $\kappa(x,\tl x)\geq B+b$.
  Taking the conditional $\eta$-expectation of $\sqrt{\frac{g(\tl
      x)}{g(x)}}$ conditioned on $Tx$ gives
  $$
  H(\eta,\tl\eta) = \int h(Tx,T\tl x) \left(\frac {g(T \tl x) \cdots
      g(T^{K-1} \tl x)} {g(T x) \cdots g(T^{K-1} x)}\right)^{1/2}\,
  d\eta(x)
  $$
  where we have
  \begin{equation}
    h(y,\tl y) = 
    \sum_{\alpha\in S} \sqrt{g(\alpha \tl y)}\sqrt{g(\alpha y)}.\label{e.hhdef}
  \end{equation}
  Since $-\log h(Tx,T\tl x)) \leq -h^g(B+b-1,1)$, we obtain the
  recursive expression
  $$ 
  -\log H(\eta,\tl\eta) \leq h^g(B+b-1,1) \cdot \left\{-\log
    H(\eta',\tl \eta')\right\},
  $$
  where $(\eta',{\tl\eta}')\in \CM^{g,g}_{B-1,b-1}$. This proves
  \qr{e.hadd}.
\end{proof}

\begin{proof}[Proof of \qr{e.hvarlog}]
  The relation \qr{e.hvarlog} follows from the Arithmetic--Geometric
  mean inequality: Fix $(x,\tl x)\in X_+\times X_0$, and assume that
  $g(\tl x) = e^{\delta(x,\tl x)} g(x)$, say, where $|\delta(x,\tl
  x)|\leq \var_{\kappa(x,\tl x)} \log g$. Then
  \begin{equation}
    \sqrt{g(\tl x)}\sqrt{g(x)} = \frac 12 \left ( g(x) + g(\tl x) \right) - 
    \delta^2 f(\delta) g(x),\label{e.fg}
  \end{equation}
  where $f$ is the continuous and strictly positive function
  $$ f(\delta) = \frac 1{\delta^2} 
  \left(\frac 12 (1+e^\delta) - e^{\delta/2}\right), $$ tending to
  $1/8$ as $\delta\to 0$.  Summing \qr{e.fg} over $y$ and $\tl y$ such
  that $(y,\tl y)=(\alpha T x,\alpha T\tl x)$, $\alpha\in S$, gives
  that
  \begin{align*}
    -\log h(Tx,T\tl x) &= -\log( 1 - \sum_{y} \delta^2(y,\tl y)
    f(\delta(y,\tl y)) g(y) ) \\
    &= (1+\ordo1) \delta^2 f(\delta),
  \end{align*}
  where $h$ as in \qr{e.hhdef}. Taking the infimum over $(Tx,T\tl x)$
  such that $\kappa(Tx,T\tl x)\geq k$ proves \qr{e.hvarlog}.
\end{proof}

\subsection{Proof of \lemref{l.renewal}}

We now use renewal theory to show \lemref{l.renewal}.  Our aim is to
prove that
$$ \PP(Y_n\leq 0)\to 0 \text{  as  } n\to \infty.$$

The Markov chain $\{Y_n\}$ will return to $0$ at random times
$\{S_0,S_1,S_2,\dots\}$ where $S_0=0$, since $Y_0=0$.  For time $n$,
define the number $N_n$ of returns as
$$ N_n = |\{k:0\leq k\leq n, Y_k=0\}| = \sup \{k: S_k \leq n \}. $$  
Define \emph{the waiting times} $T_k = S_k - S_{k-1}$ which are
independent and identically distributed waiting times due to the
Markov property of $Y_n$.  The waiting time $T_{N_n}$ is the length of
the ``cycle'' that $Y_n$ currently completes and this cycle
$Y_{S_{N_n}} \dots Y_{S_{N_n+1}}$ has length $B_\l$ for some level
$\l$. Let $L_n$ denote this level, i.e.\ $B_{L_n}=T_{N_n}$.

We now use the renewal equation to analyse
\begin{equation}\label{y}
  A_n=\PP(Y_n\leq 0).
\end{equation}
The expansion
\begin{equation}
  A_n=\PP(Y_n\leq0, N_n=1)+\PP(Y_n\leq0, N_n>1)
\end{equation}
leads to the renewal equation
\begin{equation}\label{req}
  A_n=a_n+\sum_{j=1}^{\infty} A_{n-j} p_j,
\end{equation}
where $a_n=\PP(Y_n\leq 0, N_n=1)$ and $p_j=\PP(T_1=j)$.

Let $q_\l = \PP\left\{ L_n = \l \right\}$. Then
$$
q_\l = \PP\left\{ L_n \geq \l \right\} - \PP\left\{ L_n \geq \l+1
\right\} = e^{-r_1-\dots-r_{\l-1}}\,(1 - e^{-r_{\l}}),
$$
where we use our convention that $e^{r_1-\cdots -r_{\l-1}}=1$ when
$\l=1$, i.e., $q_1=1-e^{-r_1}$.  Note that
$$p_j=\begin{cases}
  q_\l,& j=B_\l, \, \l=1,2,\dots,M-1 \\
  1-\sum_{\l=1}^{M-1} q_\l,& j=B_M \\
  0,& \text{otherwise}.\end{cases}$$ Since, $q_\l$ is the probability
that, in the first cycle, $Y_n\leq 0$ for $B_{\l-1} < n \leq
\B_\l=T_1$, we obtain
$$
a_n=\begin{cases}
  1, & n=0\\
  q_\l,& B_{\l-1}<n\leq B_\l,\, \l=1,2,\dots,M \\
  0,& \text{otherwise}.
\end{cases}
$$

It is well known that the renewal equation \eqref{req} has the
solution
\begin{equation}
  A_n=\sum_{j=0}^\infty u_{n-j}a_j,
\end{equation}
where $u_{n}=\Ex{ N_n }-\Ex{ N_{n-1} }$ and the theorem in \cite[p.\
362]{feller} states that
$$
\lim_{n\to \infty} A_n = \frac{\sum_{j=0}^\infty a_j}{\Ex{ T_1 }},
$$
provided $\sum_{j=0}^\infty |a_j| <\infty$. In our case we have
$T_1\leq B_M<\infty$ and this condition is trivially satisfied.

The ratio $\sum_j a_j\big/\Ex{T_1}$ can be transformed to that in
\qr{e.deltadef}.  We have
$$ 
\sum_{j=0}^\infty a_j = 1+\sum_{\l=1}^M b_\l q_\l = 1 + \sum_{\l=1}^M
b_\l e^{-r_1-\dots-r_{\l-1}}\,(1-e^{-r_\l}),
$$ 
and
$$ \Ex{T_1} = \sum_{\l=1}^{M-1} q_\l B_\l + (1-\sum_{\l=1}^{M-1}q_\l)B_M,$$
where the last term is due to the fact that we let $Y_{n+1}=0$
whenever $Y_n=B_M$, recall the definition of $p_j$ above.  Since $B_\l
= b_1+b_2+\dots+b_\l$, $\Ex{T_1}$ equals
$$
\sum_{\l=1}^{M-1} (e^{-r_1-\dots-r_{\l -1}}-e^{-r_1-\dots-r_\l})
(b_1+\cdots +b_\l) + e^{-r_1-\dots-r_{M-1}}(B_{M-1}+b_M)$$
$$=\sum_{\l=1}^M b_\l e^{-r_1-\dots-r_{\l -1}},$$
which is the denominator in \qr{e.deltadef}.  \qed

\subsection{Proof of the last statement in \thmref{l.bernoulli}}

We now prove the remaining statement in \thmref{l.bernoulli}: That
\qr{e.znull} implies that $\L^{* n}\mu'$ converges weakly to (the
necessarily unique) $g$-measure in $\CM^g_T$ for any initial
distribution $\mu'\in\CM(\X_+)$.

In fact, we prove convergence in the Wasserstein metric.  Given an
underlying (pseudo-) metric $d$ on the space $Y$, the corresponding
Wasserstein (pseudo-) metric $d_W$ between probability measures
$\mu,\tl\mu\in\CM(Y)$ is defined as
$$d_W(\mu,\tl\mu) := \inf_\lambda \E_\lambda\left[ d(x,\tl x)\right], $$
where the infimum is taken over all couplings $\lambda\in \CM(Y\times
Y)$ of $\mu$ and $\tl\mu$. On the space $\X_+$, we consider the
underlying metric $d(x,\tl x) = 2^{-\kappa(x,\tl x)}$ and the
corresponding Wasserstein metric $d_W$.

We already know that the condition \qr{e.znull} in
\lemref{l.bernoulli} implies that $\bar d$-distance between any pair
of $g$-chains is zero. In other words
\begin{equation}\label{e.dnoll}
  \inf_{\nu} \lim_{n\to\infty} 
  \E_\nu\left[ \mathbf{1}_{\kappa=0} (x\ii n,\tl x\ii n) \right] = 0
\end{equation}
where $\nu\in\CM(X\times X)$ signifies couplings of the two arbitrary
$g$-chains.  We shall show that \qr{e.dnoll} implies that
\begin{equation}
  \limsup_{n\to\infty}
  d_W(\L^{*n}\mu,\L^{*n}\tl\mu) = 0. \label{e.dw}
\end{equation}
Since $d_W$ metrizes the weak topology, \qr{e.dw} is equivalent to
stating that $g$ has a unique \emph{attractive} $g$-measure, i.e.\ is
a for any $\mu$, $\{\L^{*n}\mu\}$ converges weakly to a unique
$g$-measure as $n\to\infty$.

The statement \qr{e.dw} follows readily from \qr{e.dnoll}\,: Let
$N\geq 0$ be fixed but arbitrary. A coupling $\nu\in\CM(X\times X)$ of
the $g$-chains with initial distributions $\mu$ and $\tl\mu$ also
gives a coupling $\lambda = \nu\circ (x\ii n)^{-1} \otimes \nu\circ
(\tl x\ii n)^{-1}$ of $\L^{*n}\mu$ and $\L^{*n}\tl\mu$.  Since
\begin{equation*}
  d(x,\tl x) \leq 2^{-N} + \mathbf{1}_{\kappa \leq N}(x,\tl x) 
  \leq 2^{-N} + 
  \sum_{n=0}^{N-1}  \mathbf{1}_{\kappa = 0}(\T^n x,\T^n\tl x)
\end{equation*}
it therefore follows from \qr{e.dnoll} that
\begin{multline*}
  \limsup_{n}\, d_W(\L^{*n}\mu,\L^{*n}\tl\mu) \leq 2^{-N} +
  \limsup_n\, \inf_\nu\E_\nu \left[ \sum_{n=0}^{N-1}
    \mathbf{1}_{\kappa = 0}(\T^k x\ii n,\T^k\tl x\ii n)
  \right] \\
  \leq 2^{-N} + N\cdot 0.
\end{multline*}
Since $N$ was arbitrary, this concludes the proof. \qed

\end{document}